\documentclass[11pt,leqno]{article}%
\usepackage{amssymb,bbm,amsmath,amsfonts,amsthm,
amscd}%,amsbsy}
\usepackage{vmargin}
\usepackage{asymptote}
\usepackage[dvipsnames]{xcolor}
\usepackage{%hvdashln,
 array}
 \usepackage[utf8]{inputenc}

\usepackage{xspace}
\usepackage{wrapfig}
\usepackage{graphicx}
\usepackage{float}
\usepackage{subcaption}
 \usepackage{mathrsfs}
\usepackage{macros}

\theoremstyle{remark}

\usepackage[toc,page]{appendix}

\usepackage{enumerate}

\usepackage{color}

\usepackage{xcolor}

\usepackage[linktocpage,colorlinks=true]{hyperref}

\hypersetup{urlcolor=blue, citecolor=red, linkcolor=blue}

\usepackage{authblk}

\makeatletter

\@addtoreset{equation}{section}

\date{ }

\begin{document}

\title{Continuous boundary condition at the interface for two coupled fluids}

\author[1]{François Legeais\footnote{Corresponding author}}
\author[2]{Roger Lewandowski}

\affil[1,2]{IRMAR, UMR CNRS 6625, University of Rennes 1 and ODYSSEY Team, INRIA
  Rennes, France\\
E-mail: Roger.Lewandowski@univ-rennes1.fr, francois.legeais@univ-rennes1.fr}

\maketitle

\begin{abstract} We consider two laminar incompressible flows coupled by the continuous law at a fixed interface $\Ga_I$. We approach the system by one that satisfies a friction Navier law at $\Ga_I$, and we show that when the friction coefficient 
goes to $\infty$, the solutions converges to a solution of the initial system. We then write a numerical Schwarz-like coupling algorithm 
and run 2D-simulations, that yields same convergence result. 
\end{abstract}
MCS Classification: 	76D07, 35J20, 65N30
\smallskip


\section{Introduction}
We consider the two coupled fluids problem with a rigid lid assumption, given by two 3D stokes equations, 
\begin{eqnarray} \label{eq:Pb1} &&-\nu_i \Delta \uv_i +\nabla p_i = \fv_i , \quad 
\nabla \cdot \uv_i =0, \\
&&\label{eq:CL1} \uv_{1,h | \Gamma_{ I }}  = \uv_{2,h | \Gamma_{ I }}, \quad 
w_{i |\Gamma_I}=\uv_i \cdot \nv_{i | \Gamma_{ I }}=0, \\
&& \label{eq:CL2} \uv_{1 |\Gamma_1}= \uv_{2 |\Gamma_2}=0, 
\end{eqnarray}
for $i=1, 2$, where the velocities $(\uv_1, \uv_2) = (\uv_1(\x_h, z_1) , \uv_2 (\x_h, z_2)) $ are decomposed as 
$\uv_i = (\uv_{i,h}, w)$, $\uv_{i, h} = (u_{i, x}, u_{i, y})$. Moreover, $\x_h \in \mathbb{T}_2$, where $\mathbb{T}_2$ is the two dimensional torus, which means that we consider horizontal periodic boundary conditions. The interface $\Gamma_{ I }$ is given by 
$\Gamma_{ I } = \{ (\x_h, 0), \x_h \in \mathbb{T}_2 \} $, the boundaries $\Ga_i$ are given by 
$\Gamma_{1}= \{ (\x_h, z_1^+), \x_h \in \mathbb{T}_2 \}$, $\Gamma_{2 } = \{ (\x_h, z_2^-), \x_h \in \mathbb{T}_2 \} $,  
$z_1 \in J_1= [0,  z_1^+]$, $z_2 \in J_2 = [z_2^-, 0]$, where $z_1^+ >0$ and $z_2^- <0$. The coefficient $\nu_i >0$ is the viscosity of the fluid $i$, $p_i$ its pressure. 

The main characteristic of this problem is the continuity boundary condition \eqref{eq:CL1}, which is natural and physical  \cite{MR1744638}, and usually considered for 
free interfaces \cite{MR3035985}. Notice that the rigid lid assumption we consider is reasonable for laminar coupled flows, as well as for large scales. 
In this paper we adress the question of the existence and uniqueness of a weak solution to Problem \eqref{eq:Pb1}-\eqref{eq:CL1}-\eqref{eq:CL2}, given as the limit of "frictional solutions", for which we can write a numerical Schwarz-like  algorithm. More specifically, we approach this problem by the following problem  
\begin{eqnarray} \label{eq:Pb2} &&-\nu_i \Delta \uv_i +\nabla p_i = \fv_i , \quad 
\nabla \cdot \uv_i =0, \\
&&\label{eq:CLL1} \nu_i \frac{\p \uv_{i,h}}{\p \nv_i}\vert _{\Gamma_{ I }}= -\alpha (\uv_{i,h}-\uv_{j,h}), \quad 
w_{i |\Gamma_I}=\uv_i \cdot \nv_{i | \Gamma_{ I }}=0, \\
&& \label{eq:CLL2} \uv_{i |\Gamma_{i}}=0, 
\end{eqnarray}
$i, j =1,2$,  in which the continuity condition \eqref{eq:CL1} is replaced by the Navier law \eqref{eq:CLL1} where $i \not=j$. Similar problems have been already studied before, see 
\cite{MR1935990, MR2813910, MR2576510, MR1325825}, and the existence and uniqueness of a weak solution is guaranteed. We aim to investigate how Problem \eqref{eq:Pb2}-\eqref{eq:CLL1}-\eqref{eq:CLL2} approaches 
Problem \eqref{eq:Pb1}-\eqref{eq:CL1}-\eqref{eq:CL2} when the friction coefficient $\alpha$ goes to infinity. Such question  has already been adressed  in 
\cite{MR4231512} for a single fluid, where it is proved that the corresponding solution strongly converges to a solution to the corresponding Stokes (Navier-Stokes) 
equations with a no slip boundary condition when $\alpha \to \infty$.  We show in this paper the  convergence in $H^1$ space type of the solution of \eqref{eq:Pb2}-\eqref{eq:CLL1}-\eqref{eq:CLL2} to a solution of \eqref{eq:Pb1}-\eqref{eq:CL1}-\eqref{eq:CL2} (see Theorem \ref{th2.1}). 
\newline
As we shall see, numerical simulations are easily carried out by \eqref{eq:Pb2}-\eqref{eq:CLL1}-\eqref{eq:CLL2} thanks to a Schwarz-like coupling algorithm, that does not work for \eqref{eq:Pb1}-\eqref{eq:CL1}-\eqref{eq:CL2}. This method has already been successfully implemented for coupled problems, see for example \cite{Blay}.

The note is organized as follows. In the first part we set the functional framework and then we prove the convergence result, namely Theorem \ref{th2.1}. In the second part, we describe our algorithm and show some numerical results in the 2D case. In particular we check the numerical convergence of the algorithm.

\section{Convergence analysis}

\subsection{ Energy balance}

This section is devoted to the derivation of the main a priori estimate, which is standard. Let $(\uv_1, \uv_2)$ be any enough smooth solution to Problem 
\eqref{eq:Pb2}-\eqref{eq:CLL1}-\eqref{eq:CLL2}. Taking the scalar product of  equation \eqref{eq:Pb2}$_i$ by $\uv_i$ in integrating over $\mathbb{T}_2 \times J_i$ over yields by \eqref{eq:CLL2}$_i$, because of the periodic boundary conditions in the $x-y$ axes, the incompressibility condition and 
$\uv_i \cdot \nv_i = 0$ at $\Gamma_{ I }$, 
\begin{equation*}
    \nu_i \displaystyle \int_{\mathbb{T}_2 \times J_i} |\nabla \uv_i |^2 - \nu_i\displaystyle \int_{\Gamma_I} \frac{\p \uv_{i,h}}{\p \nv_i} = \displaystyle \int_{\mathbb{T}_2 \times J_i} \fv_i \cdot \uv_i,
\end{equation*}
giving by  \eqref{eq:CLL1}, 
$
    \nu_i \displaystyle \int_{\mathbb{T}_2 \times J_i} |\nabla \uv_i |^2 +\alpha \displaystyle \int_{\Gamma_I} \uv_{i,h}\cdot (\uv_{i,h}-\uv_{j,h}) = \displaystyle \int_{\mathbb{T}_2 \times J_i} \fv_i \cdot \uv_i.
$
Summing up the two equalities yields the following energy balance,  
\begin{equation}\label{eq:apriori} \begin{array}{l} 
    \nu_1 \displaystyle \int_{\mathbb{T}_2 \times J_1} |\nabla \uv_1 |^2 + \nu_2 \displaystyle \int_{\mathbb{T}_2 \times J_2} |\nabla \uv_2 |^2 +\alpha \displaystyle \int_{\Gamma_I}  |\uv_{1,h}-\uv_{2,h}|^2 = \\ 
 \hskip 8cm   \displaystyle \int_{\mathbb{T}_2 \times J_1} \fv_1 \cdot \uv_1+\displaystyle \int_{\mathbb{T}_2 \times J_2} \fv_2 \cdot \uv_2.\end{array} 
\end{equation}

\subsection{Functions spaces, variational formulation}\label{sec:func}
Let 
$
    {\cal W}_i=\{ \uv \in C^\infty (\mathbb{T}_2 \times J_i),~~\uv_{|\Gamma_i}=0,~~\uv \cdot \nv_i |_{\Gamma_I}=0,~~\nabla \cdot \uv_i =0 \},
$
equipped with 
$ || \uv ||_{i, 1} = || \g \uv ||_{L^2(\mathbb{T}_2 \times J_i)}$ which is indeed a norm due to the condition $\uv_{|\Gamma_i}=0$. Let $W_i$ denotes the completion 
of ${\cal W}_i$ with respect to this norm, 
\BEQ W = W_1 \times W_2, \quad W_0 = \{ (\uv_1, \uv_2 ) \in W, \, \, \uv_{1,h | \Gamma_I} = \uv_{2,h | \Gamma_I} \, a.e. \hbox{ in } \Gamma_I \} . \EEQ
We equip $W$ with the scalar product, for any ${\bf U} = (\uv_1,\uv_2), {\bf V} = (\vv_1, \vv_2) \in W$, 
\BEQ \Lambda ({\bf U}, {\bf V} ) =     \nu_1 \displaystyle \int_{\mathbb{T}_2 \times J_1} \nabla \uv_1 \cdot \nabla \vv_1 + \nu_2 \displaystyle \int_{\mathbb{T}_2 \times J_2} \nabla \uv_2 \cdot \nabla \vv_2.\EEQ
The space $W_0 $ is the kernel of the form $L : (\uv_1, \uv_2) \to \uv_{1,h | \Gamma_I}- \uv_{2,h | \Gamma_I}$, which is continuous by the trace theorem.  Therefore 
$W_0$ is a closed hyperplane of $W$. Let $P$ denotes the orthogonal projection over $W_0$, and ${\boldsymbol \Phi} = 
({\boldsymbol \phi}_1, {\boldsymbol \phi}_2)$ a unit orthogonal vector 
to $W_0$, so that $W_0 ^\perp = \vect {\boldsymbol \Phi}$. 

\begin{definition}(weak solution) A couple ${\bf U} = (\uv_1,\uv_2) \in W$ is a weak solution to Problem \eqref{eq:Pb2}-\eqref{eq:CLL1}-\eqref{eq:CLL2}  when $\forall \, {\bf V} = (\vv_1, \vv_2) \in W$,
\BEQ \label{eq:weak_sol} \Lambda ({\bf U}, {\bf V} ) + 
\alpha \displaystyle \int_{\Gamma_I}  (\uv_{1,h}-\uv_{2,h})\cdot (\vv_{1,h} -  \vv_{2,h} )= 
     \displaystyle \int_{\mathbb{T}_2 \times J_1} \fv_1 \cdot \vv_1+\displaystyle \int_{\mathbb{T}_2 \times J_2} \fv_2 \cdot \vv_2 = ({\bf F}, \Vv).
\EEQ
\end{definition}
Throughout the rest of the paper, we assume that $\fv_i \in  L^2(\mathbb{T}_2 \times J_i)$, $i=1,2$. 
The existence and the uniqueness of a weak solution to Problem \eqref{eq:Pb2}-\eqref{eq:CLL1}-\eqref{eq:CLL2} that satisfies the energy balance \eqref{eq:apriori} is straightforward by the Lax-Milgram Theorem for any given $\alpha >0$. Notice that work remains to be done about the pressures, by a suitable adaptation of a De Rham like theorem in this framework, which is an open problem. 

\subsection{Convergence}

Let ${\bf U}_\alpha = (\uv_1^{\alpha}, \uv_2^{\alpha}) \in W$ be the solution of \eqref{eq:Pb2}-\eqref{eq:CLL1}-\eqref{eq:CLL2}. 
We study in this section the convergence of the familly $( {\bf U}_\alpha)_{\alpha >0}$ when $\alpha \to \infty$, proving the following result. 

\begin{theorem}\label{th2.1} The familly $( {\bf U}_\alpha)_{\alpha >0}$ strongly converges in $W$ to a weak solution ${\bf U} = (\uv_1, \uv_2) \in W_0$ of Problem 
\eqref{eq:Pb1}-\eqref{eq:CL1}-\eqref{eq:CL2} when $\alpha \to \infty$, in the sense: 
\BEQ \label{eq:fv_wo} \forall \, {\bf V} = (\vv_1, \vv_2) \in W_0, \quad 
    \Lambda(\Uv,\Vv) =
  ({\bf F}, \Vv). 
\EEQ
Moreover, the solution of \eqref{eq:fv_wo}  is unique. 

\end{theorem}

\begin{proof} 
Let $\Uv^{\alpha}=(\uv_1^{\alpha}, \uv_2^{\alpha}) \in W=  W_1 \times W_2$ be the solution of $(S_1,S_2)$. We first  show that the familly $(\Uv^{\alpha})_{\alpha>0}$ is bounded in $W$. 
We have, by \eqref{eq:apriori},
\BEQ  \label{eq:apriori2} \|\Uv^{\alpha}\|_{W}^2 +\alpha \displaystyle \int_{\Gamma_I}  |\uv_{1,h}^{\alpha}-\uv_{2,h}^{\alpha}|^2  = 
\displaystyle \int_{\mathbb{T}_2 \times J_1} \fv_1 \cdot \uv_1^{\alpha}+\displaystyle \int_{\mathbb{T}_2 \times J_2} \fv_2 \cdot \uv_2^{\alpha},
\EEQ  which yields 
\begin{equation}
    \|\Uv^{\alpha}\|_{W}^2   \le \displaystyle \int_{\mathbb{T}_2 \times J_1} \fv_1 \cdot \uv_1^{\alpha}+\displaystyle \int_{\mathbb{T}_2 \times J_2} \fv_2 \cdot \uv_2^{\alpha}.
\end{equation}
We deduce from Poincaré and Cauchy-Schwarz inequalities that $(\Uv^\alpha)_{\alpha >0}$ is indeed bounded in $W$. Therefore,  we can extract a subsequence 
$(\Uv^{\alpha_n})_{n \in \N}$ ($\alpha_n \to \infty$ as $n \to \infty$) which converges weakly in $W$ to some $U \in W$. Moreover, by the trace theorem and usual Sobolev compactness results, the corresponding traces are strongly convergent in $L^2(\Gamma_I)$. As  by \eqref{eq:apriori2} 
$\lim_{n \to \infty} tr(\uv_{1, h}^{\alpha_n} - \uv_{2, h} ^{\alpha_n} ) = 0$ in $L^2(\Gamma_I)$, then $U \in W_0$. 
Finally, take  $\Vv \in W_0$ in \eqref{eq:weak_sol} as test, so that the boundary term vanishes. By passing to the limit in this case when $\alpha \to \infty$, we obtain that $U$ is a weak solution to \eqref{eq:Pb1}-\eqref{eq:CL1}-\eqref{eq:CL2}. Uniqueness is straightforward, which in addition garanties that the entire familly does converge to $\Uv$.

It remains to show the strong convergence.  Let $\lambda_{\alpha}\in \R$, be such that $\Uv^{\alpha}=P\Uv^\alpha+\lambda_{\alpha}\Phiv$ ($\Phiv$ being given in section \ref{sec:func}). We first show the strong convergence of $(P\Uv^{\alpha})_{\alpha >0}$ to $\Uv$ by taking $P\Uv^{\alpha}$ as a test  in \eqref{eq:weak_sol} which gives, by using the orthogonal decomposition of $\Uv^{\alpha}$,
\BEQ \label{eq:estim_proj}
    \Lambda(\Uv^{\alpha},P\Uv^{\alpha}) = \|P\Uv^{\alpha}\|_{W}^2 = 
 ({\bf F}, P \Uv^\alpha),
\EEQ
since the boundary  term on $\Gamma_I$ equals to zero by orthogonality. Therefore $(P\Uv^{\alpha})_{\alpha >0}$ is bounded in $W_0$, then converges weakly -{\sl up to a subsequence} (keeping the same notation)- to a limit ${\bf W}$,  strongly in $L^2((\mathbb{T}_2 \times J_1) \times (\mathbb{T}_2 \times J_2))$. Taking  $\Vv \in W_0$ in \eqref{eq:weak_sol} as test, noting that in this case
$\Lambda (\Uv^\alpha, \Vv) = \Lambda (P\Uv^\alpha, \Vv)$, and passing to the limit when $\alpha \to \infty$,  we see that ${\bf W}$
 is solution of the problem \eqref{eq:Pb1}-\eqref{eq:CL1}-\eqref{eq:CL2}, hence ${\bf W}=\Uv$ by uniqueness, and the entire sequence converges. Therefore, 
passing to the limit in \eqref{eq:estim_proj} yields 
$$\displaystyle \lim_{\alpha \to \infty} || P \Uv^\alpha ||^2_W =({\bf F}, \Uv) = \Lambda (\Uv, \Uv) = || \Uv ||_W^2,$$ 
which, together with the weak convergence, ensures the strong convergence as claimed. To conclude, it remains to prove that $\lim_{\alpha\to \infty } \lambda_\alpha = 0$. By the energy balance \eqref{eq:apriori2}, we have 
\BEQ \label{eq:limsup} \limsup_{\alpha \to \infty}  || \Uv^\alpha ||^2_W + \limsup_{\alpha \to \infty} \alpha \int_{\Ga_I} | \uv_{1, h}^\alpha - \uv_{2, h}^\alpha  |^2 = ({\bf F}, \Uv). 
\EEQ 
However, we have
$\int_{\Ga_I} | \uv_{1, h}^\alpha - \uv_{2, h}^\alpha  |^2 = \lambda^2_\alpha \int_{\Ga_I} | \boldsymbol \phi_{1, h}- \boldsymbol \phi_{1, h} |^2 $.  Therefore, since 
$(|| \Uv _\alpha ||^2_W)_{\alpha>0}$ is bounded,  by 
\eqref{eq:limsup} we have $(\alpha \lambda_\alpha^2)_{\alpha>0}$  is bounded, which can happen only if $\lambda_\alpha \to 0$ as $\alpha \to \infty$, concluding the proof. 
\end{proof}

\section{Numerical simulations}

\subsection{Algorithm}
We solve the problems  \eqref{eq:Pb2}-\eqref{eq:CLL1}-\eqref{eq:CLL2} for large values of $\alpha$,  with a coupling Schwarz like algorithm, using the software \href{http://www3.freefem.org/}{Freefem++}, for solving 2D Stokes problems by the finite element method. Our algorithm is set as follows. 
\newline
Step 1: We solve the problem on the upper part which gives a first value $\uv_1^{\alpha,0}$. 
\begin{eqnarray}
&& -\nu_1 \Delta \uv_1^{\alpha,0} + \nabla P_1^{(0)}=\fv_1,\quad 
 \nabla \cdot \uv_1^{\alpha,0} =0, \\
&& \nu_1  \frac{\p \uv_{1,h}^{\alpha,0}}{\p \nv_1}\vert _{\Gamma_{ I }}= -\alpha \uv_{1,h}^{\alpha,0},\\
&& \uv_{1,h}^{\alpha,0}\vert _{\Gamma_{ 1 }}=0,\quad 
 \uv_{1,h}^{\alpha,0} \cdot \nv_1=0.
\end{eqnarray}
This velocity allows us to solve the problem on the lower part, and to calculate the velocities step by step, up and down.

Step 2: We calculate $\uv_2^{\alpha,n}$ and  $\uv_1^{\alpha,n+1}$ by solving

 \begin{eqnarray}
&& -\nu_2 \Delta \uv_2^{\alpha,n} + \nabla P_2^{(n)}=\fv_2,\quad
 \nabla \cdot \uv_2^{\alpha,n} =0, \\
&& \nu_2  \frac{\p \uv_{2,h}^{\alpha,n}}{\p \nv_2}\vert _{\Gamma_{ I }}= -\alpha (\uv_{2,h}^{\alpha,n}-\uv_{1,h}^{\alpha,n}  ),\\
&& \uv_{2,h}^{\alpha,n}\vert _{\Gamma_{ 2 }}=0,\quad
\uv_{2,h}^{\alpha,n} \cdot \nv_2=0.
\end{eqnarray}
 and
\begin{eqnarray}
&& -\nu_1 \Delta \uv_1^{\alpha,n+1} + \nabla P_1^{(n+1)}=\fv_1, \quad  \nabla \cdot \uv_1^{\alpha,n+1} =0, \\
&& \nu_1  \frac{\p \uv_{1,h}^{\alpha,n+1}}{\p \nv_1}\vert _{\Gamma_{ I }}= -\alpha (\uv_{1,h}^{\alpha,n+1}-\uv_{2,h}^{\alpha,n}  ),\\
&& \uv_{1,h}^{\alpha,n+1}\vert _{\Gamma_{ 1 }}=0,\quad
 \uv_{1,h}^{\alpha,n+1} \cdot \nv_1=0.
\end{eqnarray}
Note that we are able to prove the stability of this algorithm, and numerical simulation confirm the convergence (see table below). 
Problem \eqref{eq:Pb1}-\eqref{eq:CL1}-\eqref{eq:CL2} cannot be solved in a similar way. Indeed, the interface conditions 
\begin{equation}
    \uv_{2,h}^{(n)}\vert _{\Gamma_{ I }}= \uv_{1,h}^{(n)}\vert _{\Gamma_{ I }},
\quad
    \uv_{1,h}^{(n+1)}\vert _{\Gamma_{ I }}= \uv_{2,h}^{(n)}\vert _{\Gamma_{ I }},
\end{equation}
imply that the sequences $(\uv_{1,h}^{(n)})_n$ and $(\uv_{2,h}^{(n)})_n$ are constant on the interface $\Gamma_{ I }$ which doesn't allow any iterations on the coupling algorithm.

\subsection{Simulation results}

We take $z_1^+=50$ , $z_2^-=-5$, $L=100$, $\nu_1=\nu_2=1$ and the source ${\bf F}= \left( \begin{pmatrix}
1 \\
-1 
\end{pmatrix},
\begin{pmatrix}
1 \\
-1 
\end{pmatrix}
\right)$ constant, for the simplicity.

\begin{figure}[h]
    \centering
\includegraphics[scale=0.5]{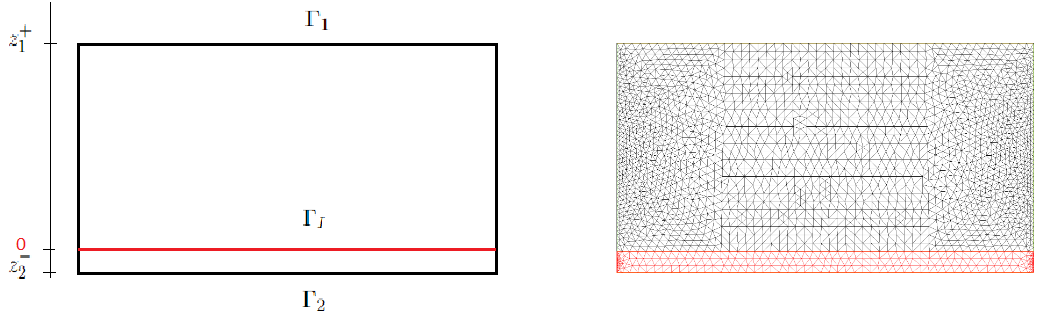}
\caption{Domain and mesh}
\end{figure}

 \begin{figure}[h!]
    \centering
    \begin{subfigure}[b]{0.4\textwidth}
        \includegraphics[scale=0.6]{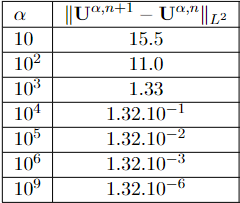} 
        
        \label{tab:error}
    \end{subfigure}
    ~
    \begin{subfigure}[b]{0.4\textwidth}
        \includegraphics[scale=0.4]{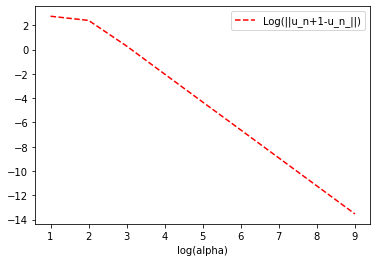} 
        
        \label{fig:log_error}
    \end{subfigure}
    \caption{L2 norm of the errors and rate of convergence for n=100}\label{fig:animals}
\end{figure}

 \begin{tabular}{|l | l | r|}
    \hline
    $\alpha$ & $n$   \\
    \hline
    $10$ & $4026$ \\
    \hline
    $10^2$ & $984$ \\
    \hline
    $10^3$ & $408$ \\
    \hline
    $10^4$ & $221$ \\
    \hline
    $10^5$ & $137$ \\
    \hline
    $10^6$ & $54$ \\
    \hline
    $10^9$ & $9$ \\
    \hline
\end{tabular}
 \begin{minipage}{.7\textwidth}%
 To check the numerical convergence of the method, we study the error term 
 $\|\Uv^{\alpha,n+1}-\Uv^{\alpha,n}\|_{L^2}$.
On the left, for a given $\alpha$, we have the first value of $n$ for which $\|\Uv^{\alpha,n+1}-\Uv^{\alpha,n}\|_{L^2}<10^{-3}$. The method is always converging, and the convergence is almost instantaneous for large $\alpha$ $(>10^6)$.

\end{minipage}%

 \begin{figure}[h!]
    \centering

        \includegraphics[scale=0.21]{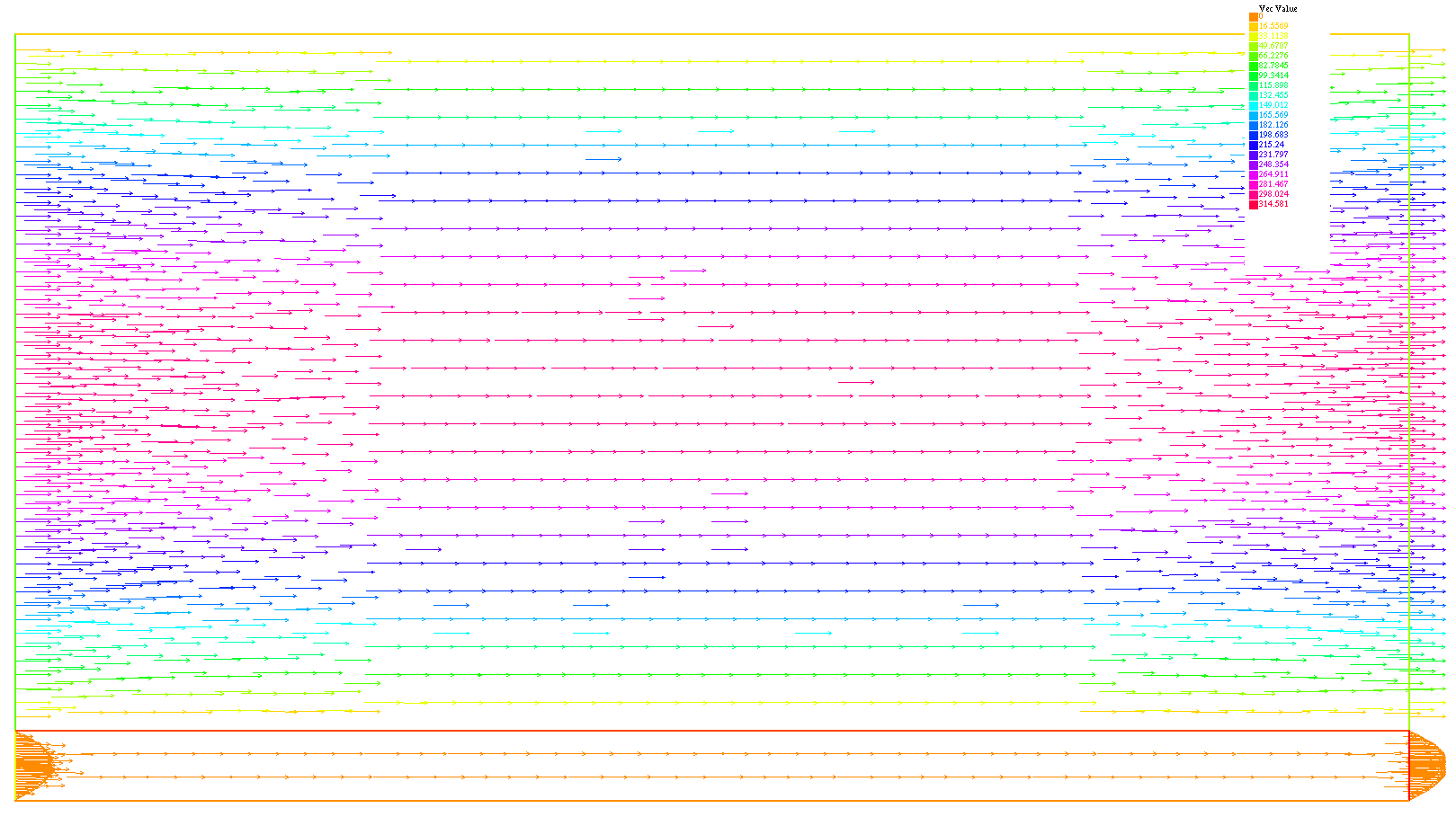}

    \caption{Velocities, $f_1=(1,-1)=f_2$, $n=9$,
$\alpha= 10^9$}
    \label{fig:press}
\end{figure}

 \bibliographystyle{plain}

\bibliography{Biblio}

\end{document}